\documentclass[a4paper,11pt]{amsart}

\usepackage{geometry}
\usepackage{amsmath}
\usepackage{amssymb}
\usepackage{amsthm}
\usepackage{amsfonts,enumerate}
\usepackage{hyperref}
\usepackage{bbm}
\usepackage{graphicx}
\usepackage{cleveref}

\usepackage{tikz}
\usetikzlibrary{cd}

\theoremstyle{definition}
\newtheorem{thm}{Theorem}
\newtheorem{lem}[thm]{Lemma}
\newtheorem{prop}[thm]{Proposition}
\newtheorem{cor}[thm]{Corollary}
\newtheorem{rem}[thm]{Remark}
\newtheorem{definition}[thm]{Definition}
\newtheorem{ex}[thm]{Example}
\newtheorem{convention}[thm]{Convention}

\crefname{thm}{Theorem}{Theorems}
\crefname{prop}{Proposition}{Propositions}
\crefname{cor}{Corollary}{Corollaries}
\crefname{lem}{Lemma}{Lemmas}

\newcommand{\N}{{\mathbb N}}
\newcommand{\Z}{{\mathbb Z}}

\DeclareMathOperator{\im}{Im}
\DeclareMathOperator{\Ker}{Ker}

\DeclareMathOperator{\Stab}{Stab}

\DeclareMathOperator{\Aut}{Aut}
\DeclareMathOperator{\Al}{Al}
\DeclareMathOperator{\Core}{Core}
\DeclareMathOperator{\As}{As}
\DeclareMathOperator{\Conj}{Conj}
\DeclareMathOperator{\Orb}{Orb}

\DeclareMathOperator{\Bij}{Bij}

\newcommand{\znz}[1]{\Z_{#1}}

\renewenvironment{abstract}
 {\par\noindent\textsc{\abstractname.}\ \ignorespaces}
 {\par\medskip}

\title[Structure groups and second homology groups of Alexander quandles]{Structure groups and second homology groups of linear Alexander quandles}
\author[Adrien Clément]{Adrien Clément}
\date{\today}



\address{Normandie Univ, UNICAEN, CNRS, LMNO, 14000 Caen, France}
\email{adrien.clement@unicaen.fr}

\begin{document}
\maketitle

\begin{abstract}
Quandles are self-distributive algebraic structures known as sources of strong knots invariants, but also appearing in other contexts.
From any quandle, one can construct two invariants: the structure group and the second quandle homology group.
These groups are useful in applications, but hard to compute.
In this paper, we focus on Alexander quandles over a cyclic group $\znz{n}$.
By using explicit rewriting techniques, we show that the structure group of such a quandle injects into $\Z^m \ltimes \znz{n}$ if $m$ is its number of orbits.
This allows us to compute its second quandle homology group, and find that the torsion part depends only on $m$ and $n$.
\end{abstract}

\section{Introduction}
Quandles were independently introduced in 1982 by Joyce  \cite{joyceClassifyingInvariantKnots1982} and Matveev \cite{matveevDistributiveGroupoidsKnot1982} as classifying invariants of knots.
A \emph{quandle} is a non-empty set $X$ equipped with a binary operation $\lhd: X\times X \to X$ that satisfies the following axioms for all $a,b,c\in X$:
\begin{enumerate}
\item idempotence: $a\lhd a = a$;
\item bijectivity of the right translation $\tau_c\colon X\to X, a\mapsto a\lhd c$;
\item self-distributivity: $(a\lhd b)\lhd c = (a\lhd c)\lhd (b\lhd c)$.
\end{enumerate}
These axioms encode the compatibility between colorings of knot diagrams and the three Reidemeister moves, so that quandles can provide coloring invariants of knots (see \cite{elhamdadiQuandlesIntroductionAlgebra2015,kamadaKnotInvariantsDerived2002} on quandles invariants).
Self-distributivity also plays a crucial role in the study of the Yang-Baxter equation \cite{brieskornAutomorphicSetsBraids1988} and the classification of Hopf algebras \cite{andruskiewitschRacksPointedHopf2003}.

Starting from a group, one may construct different quandle structures.

\begin{itemize}
\item \emph{Conjugation quandles} $\Conj(G)$ are groups $G$ with the conjugation operation
\[
g\lhd h = h^{-1} g h.
\]
\item \emph{Core quandles} $\Core(G)$ are groups $G$ with
\[
g\lhd h = h g^{-1} h.
\]
\item \emph{Alexander quandles} $\Al(M,T)$ are abelian groups $M$ equipped with a group automorphism $T$, with
\[
a\lhd b = Ta + (1-T)b.\] 
We say that $\Al(M,T)$ is \emph{linear} if $M=\znz{n}$ for an integer $n$.
\end{itemize}

A well-known example of an Alexander quandle is the Takasaki quandle (also called dihedral), where $M=\znz{n}$ and $T=-1$.

In the opposite direction, from a quandle $X$ one can construct several groups.
In this paper, we will focus on the \emph{structure group} (also called \emph{associated}, \emph{envelopping}, or \emph{adjoint group}) of a quandle $X$, which is defined as follows:
\[
\As(X) = \left< e_a, a\in X \mid e_a e_b = e_b e_{a\lhd b}, a,b\in X \right>.
\]

The structure group acts on $X$ by right translations:
\[a \cdot e_b := a\lhd b.\]

The orbits of $X$ under this action are simply called the \emph{orbits of $X$}.
A quandle $X$ is said to be \emph{connected} if the action of $\As(X)$ on $X$ is transitive.

The quadratic presentation of $\As(X)$ could suggest one to study the action of the braid group $B_k\curvearrowright X^k$, given by 
\[
\sigma_i\cdot (a_1,\ldots, a_k) = (a_1,\ldots, a_{i-1}, a_{i+1}, a_i\lhd a_{i+1}, a_{i+2}, \ldots, a_k),
\]
in order to get a normal form for elements of the structure group.
In practice, we don't know much about the image of $B_k \to \Bij(X^k)$, and so on the rewriting in the structure group.

The quandle (co)homology was introduced in 2003  \cite{carterQuandleCohomologyStatesum2003}, based on the rack homology developed in 1995 \cite{fennTrunksClassifyingSpaces1995}.
Most of the applications of quandles depend on effective computation of the \emph{second quandle homology group} $H_2^Q(X)$, and more precisely of its torsion.
The free part of $H_2^Q(X)$ is already known since 2003 for an arbitrary quandle \cite{etingofRackCohomology2003,litherlandBettiNumbersFinite2003}.
However, the torsion part is not known in most cases.
The following exceptions use quite different methods which are very difficult to generalize: odd order Takasaki quandle  \cite{niebrzydowskiSecondQuandleHomology2011}, connected quandles of order $p^2$ \cite{iglesiasExplicitDescriptionSecond2017}, connected Alexander quandle \cite{clauwensAdjointGroupAlexander2010,bakshiSchurMultipliersSecond2020}, finite abelian quandles \cite{lebedAbelianQuandlesQuandles2021}, conjugation quandles of $\overline{C}$-groups \cite{lebedConjugationGroupsStructure2024}.
The structure group provides group-theoretic tools to progress in computing $H_2^Q(X)$ through the following formula.
\begin{thm}[Eisermann \cite{eisermannQuandleCoveringsTheir2014}]
Let $X$ be a quandle.
Denote by $\Orb(X)$ the set of the orbits of $X$, and fix a family of representatives $\{a_{\mathcal{O}} \}_{\mathcal{O}\in \Orb(X)}$. Then
\begin{equation}\label{thm:eisermann}
H_2^Q(X) \cong \bigoplus_{\mathcal{O}\in \Orb(X)} \left[\Stab_{\As(X)}(a_\mathcal{O})\cap \Ker(\varepsilon) \right]_{Ab},
\end{equation}
where $\varepsilon\colon\As(X) \to \Z$ is the \emph{degree} morphism, sending each generator $e_a$ to $1$.
\end{thm}

In this paper, we focus on Alexander quandles.
Section \ref{sec:2} is devoted to constructing a surjective group morphism $f\colon\As(\Al(M,T)) \to \Z\ltimes M$. This map can be identified with $f=(\varepsilon,\omega)$ where $\varepsilon$ is the degree morphism and $\omega$ the weight map which we will define explicitly.
We also give some properties of the semi-direct product and recall several results on Alexander quandles.
In Section \ref{sec:3}, we explore properties of the $2$-cocycle related to the central group extension defined by $f$.
Then, in Section \ref{sec:4}, we prove our main theorem by using explicit rewriting techniques:
\begin{thm}
The following map is an injective group morphism:
\[\begin{array}{cccc}
u\colon& \As(\Al(\znz{n},T)) & \longrightarrow & \Z^m\ltimes \znz{n},\\
&g & \longmapsto & (Ab(g), \varepsilon(g)),
\end{array}\]
where $m$ is the number of orbits, $Ab:\As(X) \to \Z^m$ the abelianization and $\omega$ the weight map.
\end{thm}
It provides a normal form for the elements of the structure group, and finally allows us to compute the second quandle homology group of a linear Alexander quandle:
\begin{thm}
Let $X=\Al(\znz{n},T)$ be a linear Alexander quandle with $m$ orbits. The second quandle homology group of $X$ is given by
\[
H_2^Q(X) \cong \Z^{(m-1)m} \oplus \left[\gcd\left(m,\frac{n}{m}\right)\znz{n} \right]^m.
\]
\end{thm}

\section{Structure group of an Alexander quandle}\label{sec:2}
Let $M$ be an abelian group and $T\in \Aut_{Grp}(M)$.
We denote by $X:= \Al(M,T)$ the Alexander quandle $(M,\lhd)$, defined by $a\lhd b = Ta + (1-T)b$.
\subsection{Semi-direct product}
We define a right group-action of $\Z$ on $M$ by
$$a \cdot 1 = Ta.$$
Let us consider the corresponding semi-direct product group-structure $\Z \ltimes M$, given by
$$(k,a)(m,b) = (k+m, T^m a + b).$$

\begin{lem}
Let $\alpha = (k,a), \beta=(m,b)\in \Z \ltimes M$. Then
\begin{itemize}
\item The identity of $\Z \ltimes M$ is $(0,0)$.
\item The inverse of $\alpha$ is given by $\alpha^{-1} = (-k, -T^{-k} a)$.
\item The conjugate of $\alpha$ by $\beta$ is given by $\beta^{-1} \alpha \beta = (k, T^m a + (1-T^k) b)$.
\end{itemize}
\end{lem}

\begin{proof}
We check
$$(k,a)(-k,-T^{-k}a) = (k - k, T^{-k}a - T^{-k} a) = (0,0).$$
For the conjugate, we get
\[ \beta^{-1} \alpha \beta = \left ( -m, -T^{-m} b \right ) \left ( k+m, T^m a + b\right ) = \left (k, T^m a + (1-T^k) b \right ).\qedhere\]
\end{proof}

\subsection{Structure group seen as an extension}
We recall that the structure group of a quandle $(X,\lhd)$ is defined by
\[ \As(X) = \left<e_a, a\in X \mid e_a e_b = e_b e_{a\lhd b}, a,b\in X \right>. \]

\begin{prop}
The following defines a surjective group morphism:
\[ \begin{array}{cccc}
f: & \As(\Al(M,T)) & \longrightarrow & \Z \ltimes M,\\
& e_a & \longmapsto & (1,a).
\end{array} \]
\end{prop}
\begin{proof}
Begin with showing $f(e_a) f(e_b) = f(e_b) f(e_{a\lhd b})$.
For all $a,b\in M$,
\begin{align*}
f(e_a) f(e_b) & = (1,a) (1,b)\\
& = (2, Ta +b).
\end{align*}
We get
\begin{align*}
f(e_b) f(e_{a\lhd b}) & = (2, Tb + a\lhd b)\\
& = (2, Ta + b),
\end{align*}
so $f$ defines a group morphism.

Moreover, for all $(k,a)\in \Z\ltimes M$, we have
\[ f(e_0^{k-1} e_a) = (k,a) \]
hence $f$ is surjective.
\end{proof}
We denote by $\omega: \As(X) \to X$ the composition of $f$ with the projection on the second coordinate.
For $g\in \As(X)$, the quantity $\omega(g)$ will be called the \emph{weight} of $g$.
For instance,
\[
\omega(e_a e_b) = Ta + b = \omega(e_b e_{a\lhd b}).
\]
More generally,
\[
\omega(e_{a_1} \ldots e_{a_k}) = T^{k-1} a_1 + \ldots + T a_{k-1} + a_k.
\]
\begin{lem}
For all $g,h\in \As(X)$, we have
\begin{equation}\label{eq:omega_1cocycle}
\omega(gh) = T^{\varepsilon(h)} \omega(g) + \omega(h),
\end{equation}
which can be thought of as a $1$-cocycle condition.
\end{lem}

\begin{cor}\label{cor:injectivity}
Alexander quandles are injective, which means that the map
\[e_{\text{\textunderscore}} :\begin{array}{ccc}
X & \to & \As(X), \\
x & \mapsto & e_x
\end{array}
\]
is injective.
\end{cor}

\begin{proof} The composition of $e_{\text{\textunderscore}}$ with $\omega$ gives the identity map of $X$.
\end{proof}

\begin{lem}\label{lem:conjg}
For any quandle $X$, $e_x g = g e_{x \cdot g}$ for all $x\in X$ and $g\in \As(X)$.
\end{lem}

\begin{proof}
Let $x\in X$ and $g\in \As(X)$.
There exists $y_1,\ldots, y_n\in X$ and $\varepsilon_1,\ldots, \varepsilon_n\in \{\pm 1\}$ such that $g=e_{y_1}^{\varepsilon_1} \ldots e_{y_n}^{\varepsilon_n}$.
Since for all $a\in X$,
\[
e_a e_{y_i}^{\varepsilon_i} = e_{y_i}^{\varepsilon_i} e_{a\cdot e_{y_i}^{\varepsilon_i}},
\]
then we have
\[
e_x g = g e_{x\cdot g}.
\]
\end{proof}

\begin{cor}\label{cor:centerofas}
If $X$ is an Alexander quandle, then
\[
\bigcap_{x\in X} \Stab_{\As(X)}(x) = Z(\As(X)).
\]
\end{cor}

\begin{proof}
If $g\in \bigcap\limits_{x\in X} \Stab_{\As(X)}(x)$, then by \cref{lem:conjg} we have $g\in Z(\As(X))$.

If $g\in Z(\As(X))$, then $\forall x \in X, e_{x\cdot g} = g^{-1} e_x g = e_x$.
By \cref{cor:injectivity}, we get $x \cdot g = x$, and hence $g\in \Stab_{\As(X)}(x)$.
\end{proof}

\begin{rem}
The composition of $f$ with the projection on the first coordinate gives the degree morphism $\varepsilon$, so we can write $f=\varepsilon\times \omega$.
\end{rem}

If we denote the kernel of $f$ by $K(X):= \Ker(f)$, we get a group extension
\begin{equation} \label{extension} \tag{E}
\begin{array}{ccccccccc}
1 & \longrightarrow & K(X) & \longrightarrow & \As(X) & \overset{f}{\longrightarrow} & \Z \ltimes M & \longrightarrow & 1\\
& & & & g & \longmapsto & \left(\varepsilon(g),\omega(g) \right) & &
\end{array}.
\end{equation}

\subsection{Second homology group and centrality of the extension}
Recall that the structure group $\As(X)$ acts on $X$ by the right translations:
\[
a \cdot e_b = a\lhd b,\quad a,b\in X.
\]
Moreover, if $a\in X$ and $g\in \As(X)$, we have
\[
e_a g = g e_{a\cdot g}.
\]

The action of $g\in \As(X)$ is fully determined by the image of $g$ by $f$.
\begin{lem}
Let $X=\Al(M,T)$.
For all $x\in X$ and $g\in \As(X)$,
\[
x\cdot g = T^{\varepsilon(g)} x + (1-T)\omega(g).
\]
\end{lem}

\begin{proof}
We proceed by induction on the length of the word $e_{a_1}^{\varepsilon_1} \ldots e_{a_n}^{\varepsilon_n}$.
If $n=1$, then
\[
x\cdot e_{a_1}^{\varepsilon_1} = T^{\varepsilon_1} x + (1-T) \omega\left(e_{a_1}^{\varepsilon_1}\right).
\]
Now, take $n\in \N$.
Then, using \cref{eq:omega_1cocycle}, we get
\begin{align*}
x\cdot e_{a_1}^{\varepsilon_1} \ldots e_{a_n}^{\varepsilon_n} e_{a_{n+1}}^{\varepsilon_{n+1}} 
& = T^{\varepsilon_{n+1}} \left[ T^{\varepsilon_1 +\ldots \varepsilon_n} x + (1-T) \omega\left(e_{a_1}^{\varepsilon_1} \ldots e_{a_n}^{\varepsilon_n}\right) \right] + (1-T) \omega\left(e_{a_{n+1}}^{\varepsilon_{n+1}}\right) \\
& = T^{\varepsilon_1 + \ldots + \varepsilon_n + \varepsilon_{n+1}}x + (1-T) \omega\left(e_{a_1}^{\varepsilon_1} \ldots e_{a_n}^{\varepsilon_n} e_{a_{n+1}}^{\varepsilon_{n+1}} \right). \qedhere
\end{align*}
\end{proof}

\begin{prop}
The extension \eqref{extension} is central.
\end{prop}

\begin{proof}
Let $g\in K(X)$.
Then, $\varepsilon(g)=0$ and $\omega(g)=0$.
Then, for all $x\in X$,
\[
e_x g = g e_{T^{\varepsilon(g)} x + (1-T)\omega(g)} = g e_x.
\]
and hence $g$ is central.
\end{proof}

\begin{prop}\label{stabker}
Let $X=\Al(M,T)$.
For all $x_0\in X$, $\Stab_{\As(X)} (x_0)\cap \Ker(\varepsilon)$ is the abelian group $\omega^{-1}(\Ker(1-T))\cap \Ker(\varepsilon)$, and in particular does not depend on the choice of $x_0$.
\end{prop}

\begin{proof}
For all $g\in \As(X)$ and $x\in X$, we have
$$x \cdot g = T^{\varepsilon(g)} x + (1-T) \omega(g).$$
If $g\in \Stab_{\As(X)}(x_0)\cap \Ker(\varepsilon)$, then
$$ x_0 = x_0\cdot g = x_0 +(1-T) \omega(g).$$
We deduce that $\omega(g)\in \Ker(1-T)$, and hence for all $x\in X$, $g\in \Stab_{\As(X)}(x)$.
Thus, $$g\in \left(\bigcap\limits_{x\in X} \Stab_{\As(X)}(x)\right)\bigcap \Ker(\varepsilon),$$
so
$$\Stab_{\As(X)}(x_0)\cap \Ker(\varepsilon) = \left( \bigcap \limits_{x\in X} \Stab_{\As(X)}(x) \right)\bigcap\Ker(\varepsilon).$$
Since $\bigcap\limits_{x\in X} \Stab_{\As(X)}(x) =Z(\As(X))$, we get that $\Stab_{\As(X)}(x_0)\cap \Ker(\varepsilon)$ is abelian.
\end{proof}

The previous proposition allows us to simplify Eisermann's formula (\ref{thm:eisermann}) to get:
\begin{cor}
Let $X$ be an Alexander quandle.
If $m$ denotes the number of orbits of $X$, then
\[
H_2^Q(X) \cong \left(\Stab_{\As(X)}(0)\cap \Ker(\varepsilon) \right)^m.
\]
\end{cor}

In the connected case, the following gives an interpretation of the kernel $K(X)$.

\begin{lem}
The orbits of an Alexander quandle are the cosets of $\im(1-T)$.
\end{lem}

\begin{cor}
Let $X=\Al(M,T)$.
We have the following equivalence:
\[
X \text{ connected } \iff (1-T):M \to M \text{ is surjective.}
\]
\end{cor}

\begin{prop}
Let $X$ be a connected Alexander quandle. Then
\[
K(X) \cong H_2^Q(X) \cong \Stab_{\As(X)}(0)\cap \Ker(\varepsilon).
\]
\end{prop}
\begin{proof}
We have
\[\begin{array}{ccl}
\Stab_{\As(X)}(0)\cap \Ker(\varepsilon) & \overset{Prop. \ref{stabker}}{=} &\omega^{-1}( \Ker(1-T))\cap \Ker(\varepsilon)\\
&= & \omega^{-1}(0)\cap \Ker(\varepsilon)\\
&=& \Ker(\omega)\cap \Ker(\varepsilon) \\
&= &\Ker(f) \\
& = & K(X),
\end{array}\]
hence the result.
\end{proof}

\section{2-cocycle associated to the extension}\label{sec:3}
\subsection{Properties of the 2-cocycle}
Let  $X=\Al(M,T)$.
We can define a map of sets, which is a section of $f$, by
\[\begin{array}{cccc}
s\colon & \Z \ltimes M & \longrightarrow & \As(X),\\
 & (k,a) & \longmapsto & e_0^{k-1} e_a.
\end{array}\]

One studies the defect of ``morphicity'' of a section by defining a new map as follows:
\[\begin{array}{cccc}
\varphi\colon &( \Z \ltimes M) \times (\Z \ltimes M) & \longrightarrow & K(X), \\
 & \left( (k, a), (m, b)\right) & \longmapsto & s(k, a) s(m,b) s((k, a)(m, b))^{-1}.
\end{array}\]
We can give an expression in terms of the generators of $\As(X)$:

\begin{equation}\label{defphi}
\varphi((k, a),(m, b)) = \underbrace{e_0^{k-1} e_a}_{s(k, a)} \underbrace{e_0^{m-1} e_b}_{s(m, b)} \underbrace{e^{-1}_{T^ma+b} e_0^{1-k-m}}_{s((k, a)(m, b))^{-1} }. 
\end{equation}
This map turns out to be a group $2$-cocycle, that satisfies properties of symmetry and equivariance.

\begin{thm}\label{thm:propphi}
Let $X=\Al(M,T)$.
The map $\varphi$ satisfies the following properties for all $\alpha = (k, a)$, $\beta=(m, b)$ and $\gamma = (p, c)\in \Z\ltimes M$:
\begin{align}
&\varphi(\beta,\gamma) - \varphi(\alpha \beta, \gamma) + \varphi(\alpha,\beta\gamma) - \varphi(\alpha, \beta)= 0, \label{conditioncocycle}&\\
& \varphi(\alpha,(m,0))= 0, \label{conditionnormalisationd} & \\
& \varphi((k,0),\beta)= 0, \label{conditionnormalisationg} & \\
& \varphi((k,a), (m,b)) = \varphi\left((k,T a), (m,T b)\right), &\label{invariance} \\
& \varphi\left( (k,a), (m,b) \right) = \varphi\left( (1,a),(1, T^{1-m}b) \right),\\
& \varphi\left( (k,a), (m,b) \right) = \varphi\left( (m, T^{1-k} b), (k, T^m a + (1-T) b) \right). &
\end{align}
\end{thm}

\begin{definition}
The term \emph{braidings} will refer to the rewriting relations $e_a e_b = e_b e_{a\lhd b}$ for all $a,b\in X$.
\end{definition}

\begin{convention}
When it is possible, we use the additive notation when working with elements of $K(X)$.
\end{convention}

\begin{proof}
We use the centrality of $\varphi(\beta,\gamma)$ in $\As(X)$:
\begin{align*}
\varphi(\beta,\gamma) - \varphi(\alpha \beta, \gamma) & + \varphi(\alpha,\beta\gamma) - \varphi(\alpha, \beta)\\
& = \varphi(\beta,\gamma) + \varphi(\alpha,\beta\gamma) - \varphi(\alpha \beta, \gamma)  - \varphi(\alpha, \beta) \\
& = \varphi(\beta,\gamma)  + s(\alpha) s(\beta \gamma) s(\alpha \beta \gamma)^{-1} s(\alpha\beta\gamma)s(\gamma)^{-1} s(\alpha\beta)^{-1} - \varphi(\alpha,\beta) \\
& = \varphi(\beta,\gamma) + s(\alpha) s(\beta\gamma) s(\gamma)^{-1} s(\alpha\beta)^{-1} s(\alpha\beta) s(\beta)^{-1} s(\alpha)^{-1} \\
& = s(\alpha) \varphi(\beta,\gamma) s(\beta \gamma) s(\gamma)^{-1} s(\beta)^{-1} s(\alpha)^{-1} \\
& = s(\alpha) s(\beta) s(\gamma) s(\beta \gamma) s(\beta \gamma)^{-1} s(\gamma)^{-1} s(\beta)^{-1} s(\alpha)^{-1}\\
& = 0.
\end{align*}

The normalisation conditions \eqref{conditionnormalisationd} and \eqref{conditionnormalisationg} are proved by applying braidings involving the generators $e_0$.
For \eqref{conditionnormalisationd}, we have:
\[
\varphi((k,a),(m,0)) = e_0^{k-1} e_0^{m} e_{T^m a} e_{T^m a}^{-1} e_0^{1-k-m} = 0.
\]

The invariance by multiplication by $T$ \eqref{invariance} uses the centrality of $\varphi(\alpha,\beta)$ in $\As(X)$ and braidings of the generators:
\begin{align*}
\varphi((k,a), (m,b)) & = e_0^{-1} e_0 \varphi((k,a), (m,b)) \\
& = e_0^{-1} \varphi((k,a), (m,b)) e_0 & \text{centrality} \\
& = \varphi((k,Ta), (m,Tb)) & s(k,a)e_0 = e_0 s(k,Ta)
\end{align*}

The reduction to degree $1$ follows from the next lemma.

Since $\varphi((1,a),(1,b)) = \varphi((1,b), (1,a\lhd b)$, we get the last point.
\end{proof}

\begin{lem}\label{equivariancedegre1}
For all $(k, a), (m, b)\in \Z\ltimes M$, we have
\begin{itemize}
\item $\varphi((k, a),(m, b)) = \varphi((1, T^{1-k}a),(m, T^{1-k}b)) = \varphi( (1,a), (m,b) )$,
\item $\varphi((k, a),(m, b)) = \varphi((k, a),(1, T^{1-m}b))$.
\end{itemize}
\end{lem}

\begin{proof}
We only need to push in \eqref{defphi} every generator $e_0$ to the left.
For all $(k,a), (m,b)\in \Z\ltimes X$,
\[\varphi((k,a),(m,b)) = e_0^{-1} e_{T^{-k}a} e_{T^{-k} (T^{1-m} b)} e^{-1}_{T^{-k} (Ta + T^{1-m} b)}. \qedhere
\]
\end{proof}

The group $2$-cocycle $\varphi$ generates the whole kernel of the extension:
\begin{prop}
We have
\[
K(\Al(M,T)) = \left< \im(\varphi) \right> .
\]
\end{prop}

\begin{proof}
Let $g\in K(\Al(M,T))$.
Since $\varepsilon(g)=0$, after applying the braidings, one can write $g$ in the form
\[
g= \underbrace{e_{x_1}\ldots e_{x_k}}_{g^+} ( \underbrace{e_{y_1} \ldots e_{y_k}}_{g^-} )^{-1}.
\]
We will apply the following relations:
\[
\forall a,b\in M, \qquad e_a e_b = e_0 e_{Ta+b} \varphi((1,a),(1,b)).
\]
Denote

\[
\begin{array}{ll}
\phi(x_1, x_2, \ldots,x_k):= & \varphi((1,x_1),(1,x_2)) \varphi((1,Tx_1+x_2),(1,x_3)) \ldots\\
& \ldots \varphi( (1,\omega(e_{x_1}\ldots e_{x_{k-1}})), (1,x_k))\in \left<\im(\varphi)\right>.
\end{array}
\]

We get
\[
g=e_0^{k-1} e_{\omega(g^+)} \phi(x_1,\ldots, x_k) \left(e_0^{k-1} e_{\omega(g^-)} \phi(y_1,\ldots,y_k)\right)^{-1}.
\]
Since $\varphi(\alpha,\beta)$ is central for all $\alpha,\beta\in \Z\ltimes M$, we have
\[
g = e_0^{k-1} e_{\omega(g^+)} e^{-1}_{\omega(g^-)} e_0^{1-k} \phi(x_1,\ldots, x_k) \phi(y_1,\ldots,y_k)^{-1}.
\]
But $\omega(g) = 0$, so from the last equality, $\omega(g^+)=\omega(g^-)$.
Thus,
\[
g= \phi(x_1,\ldots, x_k) \phi(y_1,\ldots,y_k)^{-1} \in \left<\im(\varphi)\right>. \qedhere
\]
\end{proof}

We next explore how the section $s$ behaves with respect to the conjugation in $\Z\ltimes M$.

\begin{lem}
For $\alpha=(k,a), \gamma=(p,c)\in \Z\ltimes M$, we have
\[
s\left (\gamma^{-1} \alpha \gamma\right ) = s\left (k, T^p a + (1-T^k)c \right ).
\]
\end{lem}

In degree $0$ or $1$, this leads to:
\begin{itemize}
\item If $k=p=0$, $s\left (\gamma^{-1} \alpha \gamma\right ) = s\left (0, a \right ) = s(\alpha).$
\item If $k=p=1$, $s\left (\gamma^{-1} \alpha \gamma\right ) = s\left (1, T a + (1-T)c \right ) = s\left (1, a\lhd c\right ).$
\end{itemize}

Regarding this compatibility with the conjugation, the section $s$ turns out to be not that far from being a quandles morphism.
\begin{lem}
Let $\alpha=(1,a),\gamma=(p,c)\in \Z\ltimes M$.
Then,
\[
s\left (\gamma^{-1} \alpha \gamma \right ) = s \left (\gamma \right )^{-1} s(\alpha ) s(\gamma).
\]
\end{lem}

\begin{proof}
We check
\begin{align*}
s \left (\gamma \right )^{-1} s(\alpha ) s(\gamma) & = e_c^{-1} e_0^{1-p} e_a e_0^{p-1} e_c\\
& = e_c^{-1} e_{T^{p-1} a} e_c \\
& = e_{T^{p-1} a\lhd c} \\
& = e_{T^p a +(1-T)c} \\
& = s \left ( \gamma^{-1} \alpha \gamma\right ).\qedhere
\end{align*}
\end{proof}

\subsection{Cocyle restricted to degree 0}
The inclusion $\iota_0\colon \begin{array}{lcl}
M & \longrightarrow & \Z\ltimes M, \\
a & \longmapsto & (0,a).
\end{array} $ allows us to restrict the group $2$-cocycle $\varphi$ to a map
\[
\varphi_0\colon \begin{array}{lcl}
M \times M & \longrightarrow & K(X),\\
(a,b) & \longmapsto & \varphi( (0,a), (0,b)).
\end{array}
\]

Restricting \cref{thm:propphi} to degree $0$, we obtain:
\begin{thm}
Let $X=\Al(M,T)$.
The map $\varphi_0$ satisfies the following properties for all $a,b,c\in M$:
\begin{align}
&\varphi_0(b,c) - \varphi_0(a+b, c) + \varphi_0(a,b+c) - \varphi_0(a, b)= 0, \label{conditioncocycle0}&\\
& \varphi_0(a,0)= 0, & \\
& \varphi_0(0,b)= 0, & \\
& \varphi_0(a, b) = \varphi_0(T a, T b). &
\end{align}
\end{thm}

\begin{prop}\label{tressage0}
Let $X=\Al(M,T)$.
For all $a,b\in M$, we have
\[
\varphi_0(a,b) = \varphi_0(Tb, a + (1-T) b)
\]
\end{prop}

\begin{proof}
\begin{align*}
\varphi_0(a,b) & = e_0^{-1} e_a e_0^{-1} e_b e_{a+b}^{-1} e_0 \\
& = e_0^{-1} e_0^{-1} e_{T^{-1} a}  e_b e_{a+b}^{-1} e_0 \\
& = e_0^{-1} e_0^{-1} e_b e_{ a + (1-T)b}   e_{a+b}^{-1} e_0 \\
& = e_0^{-1} e_{Tb} e_0^{-1} e_{ a + (1-T)b} e_{a+b}^{-1} e_0. \qedhere
\end{align*}
\end{proof}

In the connected case, Clauwens (\cite{clauwensAdjointGroupAlexander2010}) proved the bilinearity of the inverse of $\varphi_0$.
In the same spirit, we study
\begin{align*}
\lambda(x,y) & = \varphi_0(y,x) - \varphi_0(x,y) \\
& = e_0^{-1} e_y e_0^{-1} e_x e_{x+y}^{-1} e_0 e_0^{-1} e_{x+y} e_y^{-1} e_0 e_x^{-1} e_0 \\
& = \left[ e_0^{-1} e_y, e_0^{-1} e_x \right],
\end{align*}
which will be helpful to explore properties of $\varphi_0$.

\begin{prop}
Let $X=\Al(M,T)$.
For all $u,v,w\in M$, we have
\begin{itemize}
\item $\lambda(u,v) = \varphi_0( (1-T)v, u) = \varphi_0(u, (1-T) T^{-1} v)$
\item $\lambda(u,v) = -\lambda(v,u)$
\item $\lambda(u+v,w) = \lambda(u,w) + \lambda(v,w)$
\item $\lambda(u,v+w) = \lambda(u,v) + \lambda(u,w)$
\item $\lambda(u,u) = 0$
\item $\lambda(Tu,Tv) = \lambda(u,v)$.
\end{itemize}
\end{prop}

We do not suppose that $1-T$ is invertible at this point.

From \cref{tressage0}, we have
\[
\varphi_0(u,v) + \varphi_0((1-T)v,u) = \varphi_0(Tv, u+(1-T) v) + \varphi_0((1-T)v,u).
\]
Using the cocycle condition with $x= Tv, y= (1-T)v$, $z = u$ and \cref{tressage0}, we get
\[
\varphi_0(u,v) + \varphi_0((1-T)v,u) = \varphi_0(v,u) + \varphi_0(Tv, (1-T)v) =  \varphi_0(v,u),
\]
since $\varphi_0(Tv, (1-T)v) = \varphi_0(0,v) = 0$.

Hence, we have
\begin{equation} \label{lienlambdaphi}
\lambda(u,v) = \varphi_0((1-T) v,u).
\end{equation}

The cocycle condition for $\varphi_0$ with $x = (1-T) c, y =(1-T) b$ and $z= a$ leads to
\begin{equation}
\varphi_0((1-T)b, a) + \varphi_0((1-T)c, a + (1-T) b) = \varphi_0 ((1-T) (b+c), a) + \varphi_0 ((1-T) c, (1-T) b).
\end{equation}

Hence, using \eqref{lienlambdaphi} we have
\begin{equation}\label{lambdaintermediaire}
\lambda(a,b) + \lambda(a+ (1-T)b, c) = \lambda(a, b+c) + \lambda((1-T)b, c).
\end{equation}

Let us proceed in two steps.
First, the sum of
\begin{itemize}
\item $\varphi_0(a,b+c) + \varphi_0(b,c) = \varphi_0(a+b,c) + \varphi_0(a,b)$
\item $\varphi_0(c,b+a) + \varphi_0(b,a) = \varphi_0(c+b,a) + \varphi_0(c,b)$
\end{itemize}
gives
\[
\lambda(a+b,c) + \lambda(a,b) = \lambda(a, b+c) + \lambda(b,c).
\]

Then, by substracting this equality from \eqref{lambdaintermediaire}, we get
\begin{equation}
\lambda(a+(1-T)b,c) - \lambda(a+b,c) = \lambda((1-T)b,c) - \lambda(b,c).
\end{equation}

Since the right-hand member does not depend on $a$, it is equal to the left-hand member specified in $a=-b$, so
\[
\lambda(a + (1-T)b,c) - \lambda(a + b,c) = \lambda(-Tb,c) .
\]

Writing $\alpha = a +b$ and $\beta =-Tb$, we get
\[
\lambda(\alpha + \beta,c) = \lambda(\alpha,c) + \lambda(\beta,c)
\]
for all $\alpha,\beta,c\in M$.

We have shown that $\lambda$ is additive in the first coordinate.

Further, since $\lambda(x,y) = \varphi_0(y,x) -\varphi_0(x,y) = -\lambda(y,x)$, we also have the additivity of $\lambda$ in the second coordinate.

The bi-additivity gives another expression of $\lambda$ in terms of $\varphi$:
\begin{align*}
\varphi_0(\alpha, (1-T) \beta) & = \varphi_0((1-T) T \beta, \alpha + (1-T)^2 \beta) \\
& = \lambda(\alpha + (1-T)^2 \beta, T \beta) \\
& = \lambda(\alpha, T\beta) + \lambda((1-T)^2\beta, T\beta) \\
& = \varphi_0( (1-T) T\beta, \alpha) + \varphi_0( (1-T)T\beta, (1-T)^2\beta) \\
& = \varphi_0( (1-T) T\beta, \alpha) + \underbrace{\varphi_0(0, (1-T)\beta )}_{=0} \\
& = \lambda(\alpha, T\beta)\\
& = -\lambda(T\beta, \alpha)\\
& = -\varphi_0((1-T)\alpha, T\beta).
\end{align*}

\section{Linear Alexander quandles}\label{sec:4}
\begin{definition}
Quandles of the form $\Al(\znz{n}, T)$ are called \emph{linear Alexander quandles}.
\end{definition}

In this section, let $X = \Al(\znz{n}, T)$.
It is known that the orbits of an Alexander quandle $X$ are the cosets of $\im(1-T)$.
Since $\im(1-T)$ is a subgroup of $\znz{n}$, it is cyclic and hence there exists a unique divisor $q\in \N$ of $n$ such that $\im(1-T) = q\znz{n}$.
But $\left.\raisebox{.2em}{$\znz{n}$}\middle/\raisebox{-.2em}{$q \znz{n}$}\right. \cong \znz{q}$, so we get $q = m$: the number of orbits.

Denote by 
\[
\begin{array}{cccc}
Ab:&\As(X) &\to & \As(X)_{Ab} \\
& e_x & \mapsto & [e_x]
\end{array}
\]
the abelianization.
For $x\in X$ and $g\in \As(X)$, the relation $g^{-1} e_x g = e_{x \cdot g}$ implies $[e_x] = [e_{x\cdot g}]$, and $\As(X)_{Ab} \cong \Z^m$.

\subsection{Normal form}
We now prove our main theorem.
It will give us a normal form for the elements of the structure group of a linear Alexander quandle.

\begin{thm}\label{thm:normalform}
Let $X$ be a linear Alexander quandle over $\znz{n}$.
The following map is an injective group morphism:
$$u\colon\begin{array}{ccc}
\As(X) & \longrightarrow & \Z^m \ltimes \znz{n} \\
g & \longmapsto & (Ab(g), \omega(g))
\end{array},$$
where $m$ is the number of orbits of $X$, $Ab\colon \As(X) \to \Z^m$ the abelianization and $\omega$ the weight map.
\end{thm}
Since we have a surjective homomorphism $\Z^m\to \Z$, the action of $\Z$ on $\znz{n}$ gives us an action of $\Z^m$ on $\znz{n}$:
\[
a \cdot [g]= a \cdot \varepsilon(g) = T^{\varepsilon(g)} a.
\]

In the proof, we wil need the following technical result:
\begin{lem}
For all $\alpha, \beta, \gamma \in X=\Al(\znz{n},T)$, we have
\begin{equation}\label{relationlinear}
e_\alpha e_\beta = e_{\alpha-(1-T)\gamma} e_{\beta + (1-T)T\gamma}
\end{equation}
\end{lem}
\begin{proof}
By using the bi-additivity of $\lambda$, we have for all $a,b\in X, \lambda(a,b) = ab\lambda(1,1) = 0$.
Then, $\varphi_0((1-T)a,b) = \lambda(a,b) = 0$.
Remember that the cocycle condition for $\varphi_0$ is
\[
\varphi_0(y,z) + \varphi_0(x,y+z) = \varphi_0(x+y,z) + \varphi_0(x,y).
\]
By replacing $x=a$, $y=-(1-T)c$ and $z=b+(1-T)c$, we get
\[
\underbrace{\varphi_0(-(1-T)c,b+(1-T)c)}_{=0} + \varphi_0(a,b) = \varphi_0(a-(1-T)c, b+(1-T)c) + \underbrace{\varphi_0(a,-(1-T)c)}_{=0}
\]
and hence
\[
\varphi_0(a,b) = \varphi_0(a-(1-T)c, b+(1-T)c).
\]
The expression of $\varphi_0(a,b)$ in terms of the generators of the structure group gives
\[
e_0^{-2} e_{T^{-1}a} e_b = e_0^{-2} e_{T^{-1}a - (1-T)T^{-1}c} e_{b+(1-T)c}.
\]
We finally get, for all $\alpha,\beta,\gamma\in X$,
\[
e_\alpha e_\beta = e_{\alpha-(1-T)\gamma} e_{\beta + (1-T)T\gamma}.\qedhere
\]
\end{proof}

\begin{proof}[Proof of Theorem \ref{thm:normalform}]
First, let us show that $u$ is well-defined.
Let $a,b\in X$.
On one hand, we have $u(e_a) u(e_b) =([e_a] + [e_b], Ta +b)$.
On the other hand, $u(e_b) u(e_{a\lhd b}) = ([e_b] + [e_{a\lhd b}],Ta +b)$.
But $a\lhd b$ and $a$ belong to the same orbit, and since $\Z^m$ is commutative, we have
\[
u(e_a) u(e_b) = u(e_b) u(e_{a\lhd b}).
\]

We fix the family $\{0,1, \ldots, m-1\}$ of representatives of the orbits of $X=\Al(\znz{n},T)$.

We will use the following result in the rewriting process:
\begin{prop}[\cite{lebedStructureGroupsSetTheoretic2019}] \label[prop]{prop:centralitepuissance}
For all $x\in \Al(\znz{n},T)$, there exists $d(x)\in \N$ such that $e_x^{d(x)}$ is central.
Moreover, if $x$ and $y$ belong to the same orbit, $d(x)=d(y)$ and $e_x^{d(x)} = e_y^{d(y)}$.
\end{prop}

Now, take $g\in \As(X)$, and fix a word $\tilde{w} = e_{\tilde{a}_1}^{\tilde{\eta}_1} \ldots e_{\tilde{a}_k}^{\tilde{\eta}_k}$ representing $g$.
We can rewrite it into a new word $w_f = e_{m-1}^{\varepsilon_{m-1}} \ldots e_1^{\varepsilon_1} e_0^{\varepsilon_0-1} e_{(1-T)c}$ which also represents $g$, such that the rewriting path follows this order:
\begin{enumerate}
\item Using the braiding relations, we build blocks of generators whose colours belong to the same orbit:
\[
\tilde{w} \to \underbrace{e_{a_1}^{\eta_1} \ldots e_{a_i}^{\eta_i }}_{=\omega_{m-1}} \ldots \underbrace{e_{a_j}^{\eta_j} \ldots e_{a_k}^{\eta_k }}_{=\omega_0} = w.
\]
The numbers $\varepsilon_s:= \varepsilon(\omega_s)$ will be the powers involved in the writing of $w_f$.
\item We reduce to positive powers of generators.
Let $\xi_p\in \N$ be the smaller integer such that $\xi_p d(p) + \eta_p \geq 0$, where $d(p)$ is the smallest power such that $e_{p}^{d(p)}$ is central. 
We now use \cref{prop:centralitepuissance} to reduce to positive powers of generators:
\[
w \to e_0^{-d(0)[\xi_j + \ldots + \xi_k]} \ldots e_{m-1}^{-d(m-1)[\xi_1+ \ldots + \xi_i]} \underbrace{e_{a_1}^{\xi_1 d(a_1) + \eta_1} \ldots e_{a_k}^{\xi_k d(a_k) + \eta_k}}_{= \omega^+}
\]

\item We recursively apply the relation \eqref{relationlinear} on $w^+$ from left to right, which leads to 
\[
w^+ \to e_{m-1}^{\tilde{\varepsilon}_{m-1}} \ldots e_1^{\tilde{\varepsilon}_1} e_0^{\tilde{\varepsilon}_0-1} e_{(1-T)c} 
\]
with $c\in X$.
\item We conclude by using again \cref{prop:centralitepuissance}: we have $\tilde{\varepsilon_r} = \varepsilon_r + d(r)\sum\limits_{a_u\in r +(1-T)X} \xi_u$, and since $e_r^{d(r)}$ is central, we get 
\[
e_{m-1}^{\tilde{\varepsilon}_{m-1}} \ldots e_1^{\tilde{\varepsilon}_1} e_0^{\tilde{\varepsilon}_0-1} e_{(1-T)c}
\to
e_{m-1}^{d(m-1)[\xi_1+ \ldots + \xi_i]} \ldots e_0^{d(0)[\xi_j + \ldots + \xi_k]}  e_{m-1}^{\varepsilon_{m-1}} \ldots e_1^{\varepsilon_1} e_0^{\varepsilon_0-1} e_{(1-T)c}.
\]
It allows us to kill the central terms that we artificially added.
\end{enumerate}

We get the final form for $g\in \As(X)$:
\begin{equation} \label{formenormale}
g=e_{m-1}^{\varepsilon_{m-1}} \ldots e_{1}^{\varepsilon_1} e_0^{\varepsilon_0 -1} e_{(1-T)c},
\end{equation}
where $\varepsilon_{m-1},\ldots, \varepsilon_0$ and $(1-T)c$ are uniquely determined by $Ab(g)$ and $\omega(g)$.
Indeed, $(\varepsilon_{m-1},\ldots, \varepsilon_0)$ can be identified with $Ab(g)$ since each component is a partial degree of $g$.
By using the $1$-cocycle condition satisfied by $\omega$ (\ref{eq:omega_1cocycle}), we have
\[
\omega(g) = T\omega\left(e_{m-1}^{\varepsilon_{m-1}} \ldots e_1^{\varepsilon_1} e_0^{\varepsilon_0-1}\right) + (1-T)c
\]
Moreover, the expression (\ref{eq:omega_1cocycle}) shows that $(1-T)c$ depends only on $Ab(g)$ and $\omega(g)$.
Hence \eqref{formenormale} yields a normal form for $g$.

Hence, we have an injective group morphism
\[
\begin{array}{ccc}
\As(X) & \longrightarrow & \Z^m \ltimes M\\
g & \longmapsto & \left( (\varepsilon_{m-1}, \ldots, \varepsilon_1, \varepsilon_0), \omega(g) \right)
\end{array} \qedhere
\]
\end{proof}

We can give a nice expression of $\omega(g)$:

\begin{equation}\label{omega}
\omega(g) =  \underbrace{\frac{T^{\varepsilon_{m-1}} -1}{T-1} T^{\varepsilon_0+\ldots+ \varepsilon_{m-2}} (m-1) + \ldots + \frac{T^{\varepsilon_1} -1}{T-1}T^{\varepsilon_0} 1}_{b(g)} + (1-T)c,
\end{equation}
where $\frac{T^{k}-1}{T-1}$ denotes the evaluation of the polynomial $\frac{X^k-1}{X-1}$ at $X=T$.

Note that $(1-T)c$ is determined uniquely, whereas $c$ is not in the non-connected case.

\subsection{Second homology group}
The aim of this subsection is to compute the second homology group of a linear Alexander quandle.
To begin with, we can restrict the group morphism $u$ from \cref{thm:normalform} to $\Ker(\varepsilon)$.
\begin{prop}
The kernel $\Ker(\varepsilon)$ is an abelian subgroup of $\Z^m\ltimes \znz{n}$.
\end{prop}

\begin{proof}
Let $g,h\in \Ker(\varepsilon)$.
By the injectivity of $u$, we can identify $g=(Ab(g), \omega(g))$ and $h=(Ab(h), \omega(h))$.
Then
\begin{align*}
gh & = (Ab(g), \omega(g))(Ab(h), \omega(h))\\
& = (Ab(g) + Ab(h), T^{\varepsilon(h)} \omega(g) + \omega(h))\\
& = (Ab(g) + Ab(h), \omega(g) + \omega(h))\\
& = hg.\qedhere
\end{align*}
\end{proof}

\begin{prop}
Let $X = \Al(\znz{n},T)$ be a linear Alexander quandle with $m$ orbits.
There exists a unique group morphism $\overline{\omega}$ that makes the following diagram commute:

\begin{center}
\begin{tikzcd}
\Ker(\varepsilon) \arrow[r,twoheadrightarrow, "\alpha"] \arrow[d, twoheadrightarrow, "\omega"] & \Z^{m-1} \arrow[d, dashed, "\overline{\omega}"] \\
\znz{n} \arrow[r, twoheadrightarrow, "p"] &\znz{m}
\end{tikzcd}
\end{center}
where $\alpha= Ab_{\vert \Ker(\varepsilon)}$ and $p$ is the projection on the set of orbits.
\end{prop}

\begin{proof}
We check that $\alpha$ and $\omega_{\vert \Ker(\varepsilon)}$ are surjective.
For $(\varepsilon_{m-1}, \ldots, \varepsilon_1)\in \Z^{m-1}$ we consider $h:= e_{m-1}^{\varepsilon_{m-1}} \ldots e_1^{\varepsilon_1} e_0^{-\sum\limits_{i=1}^{m-1} \varepsilon_i} \in \Ker(\varepsilon)$. Then, $\alpha(h) = (\varepsilon_{m-1}, \ldots, \varepsilon_1)$, so $\alpha$ is surjective.
Now, for every $x\in X$, we have $\omega(e_0^{-1} e_x) = x$, so $\omega_{\vert \Ker(\varepsilon)}$ is surjective.

Let $g=\left( (\varepsilon_{m-1}, \ldots, \varepsilon_0), \omega_g\right)\in\Ker(\varepsilon)$.
We have $\varepsilon_0= -\sum\limits_{i=1}^{m-1} \varepsilon_i$, so we can identify $\alpha(g)$ with $(\varepsilon_{m-1}, \ldots, \varepsilon_1)$.

Let $c\in \Z^{m-1}$.
Take $\widehat{c}$ an $\alpha$-preimage of $c$, and define $\overline{\omega}(c):= p( \omega(\widehat{c}) )$.
The weights of two preimages of $c$ differ only by an element $(1-T)X$.
Recall that
\[
b(g)=\frac{T^{\varepsilon_{m-1}} -1}{T-1} T^{\varepsilon_0+\ldots+ \varepsilon_{m-2}} (m-1) + \ldots + \frac{T^{\varepsilon_1} -1}{T-1}T^{\varepsilon_0} 1\]
depends only on $\alpha(g)$, and $\omega(g)=b(g) + (1-T)c$ for an element $c\in X$.
We have $p((1-T)c) = 0$ so $\overline{\omega}$ does not depend on the choice of the preimage.
Moreover, $\overline{\omega}$ is a well-defined group morphism since $\znz{m}$ is abelian.
It is surjective, as $c$ is sent to the representative of the corresponding orbit, which are viewed as the non-trivial elements of $\znz{m}$.
The surjectivity of $\alpha$ gives the commutativity of the square, and the uniqueness of the completion.
\end{proof}

\begin{cor}
For an Alexander quandle on $\znz{n}$ with $m$ orbits, the kernel of the degree morphism $\Ker(\varepsilon)$ is the pullback
\[
\Ker(\varepsilon) \cong \Z^{m-1} \underset{\overline{\omega},\znz{m}, p}{\times} \znz{n}
\]
\end{cor}
The above pullback will be denoted $\Z^{m-1} \underset{\znz{m}}{\times} \znz{n}$ for brevity.

We get a nice expression for the second homology group of a finite linear Alexander quandle.
\begin{cor}\label{cor:H2}
Let $X=\Al(\znz{n},T)$ be a linear Alexander quandle with $m$ orbits.
The second quandle homology group of $X$ is given by
\[
H_2^Q(X) \cong \left[ \Z^{m-1} \underset{\znz{m}}{\times} \Ker(1-T) \right]^m.
\]
\end{cor}

We can give an explicit description of the pullback in terms of cyclic groups.
\begin{thm}
With the same notation as \cref{cor:H2}, we have
\[
H_2^Q(X) \cong \Z^{(m-1)m} \oplus \left[\gcd\left(m,\frac{n}{m}\right)\znz{n} \right] ^m.
\]
\end{thm}

\begin{proof}
We denote $\im(1-T)$ by $m\znz{n}$.
\begin{align*}
\Z^{m-1} \underset{\znz{m}}{\times} \Ker(1-T) & \cong
\left\{ ((\varepsilon_{m-1}, \ldots, \varepsilon_0), a) \mid
\begin{array}{lcl}
\varepsilon_i\in \Z, a\in \Ker(1-T), \sum \varepsilon_i=0,\\
\sum \varepsilon_i\cdot i \equiv a \pmod{ \im(1-T) }
\end{array}
\right\} & \\
& \cong 
\left\{ ((\varepsilon_{m-1}, \ldots,\varepsilon_2, \varepsilon_1'), a) \mid
\begin{array}{lcl}
\varepsilon_i\in \Z, a\in \Ker(1-T),\\
\varepsilon_1' \equiv a \pmod{ \im(1-T) }
\end{array}
\right\} \\
&\cong \Z^{m-2} \oplus \left(\Z \underset{\znz{m}}{\times} \Ker(1-T) \right)\\
&\cong \Z^{m-2} \oplus \left(\Z \underset{\znz{m}}{\times} \frac{n}{m}\znz{n} \right)\\
&\cong \Z^{m-2} \oplus \left\{(\alpha, a) \mid \alpha\in \Z, a\in \frac{n}{m}\znz{n}, \alpha\equiv a \pmod {m\Z} \right\} \\
&\cong \Z^{m-2} \oplus \left\{ (\alpha, a) \mid \alpha \in \Z, a\in \frac{n}{m}\znz{n}, a\in m\znz{n} \right\}\\
&\cong \Z^{m-1} \oplus \left( \frac{n}{m}\znz{n}\cap m\znz{n} \right) \\
&\cong \Z^{m-1} \oplus \left(\gcd\left(m,\frac{n}{m}\right)\znz{n} \right),
\end{align*}
where $\varepsilon_1'=\sum \varepsilon_i\cdot i$.
\end{proof}

\begin{ex}
~
\begin{enumerate}
\item Suppose that $n=p$ is prime.
Then $1-T$ is invertible and hence $H_2^Q(X) = \{0 \}$.
\item Suppose that $n=pq$ is a product of two distinct primes.
If $1-T$ is not invertible, then $m$ is either $p$ or $q$.
If $m=p$, we have $\frac{n}{m} = q$, so $H_2^Q(X) = \{0\}$.
\item Suppose that $n=p^2$ is the square of a prime.
If $1-T$ is not invertible, then $m =p$ hence $H_2^Q(X) \cong \Z^{(p-1)p} \oplus \left( \znz{p}\right)^{p}$.
\end{enumerate}
\end{ex}

\bibliographystyle{alphaurl}
\bibliography{article1_final}
\end{document}